\documentclass{amsart}
\usepackage[utf8]{inputenc}
\usepackage{subcaption}
\usepackage{amsmath}
\usepackage{amssymb}
\usepackage{amsfonts}
\usepackage{amsthm}
\usepackage{mathrsfs}
\usepackage[all]{xy}
\usepackage{graphicx}
\usepackage{color}
\usepackage{cite} 
\usepackage{url}
\usepackage{indent first}
\usepackage[labelfont=bf,labelsep=period,justification=raggedright]{caption}
\usepackage[english]{babel}
\usepackage[utf8]{inputenc}
\usepackage{hyperref}
\usepackage[colorinlistoftodos]{todonotes}
\usepackage{tkz-fct}
\usepackage{tikz}
\usetikzlibrary{calc}
\usepackage[enableskew]{youngtab}
\usepackage{pstricks}
\usepackage{multicol}
\usepackage{graphicx} 
\usepackage{fancyhdr}
\theoremstyle{plain}
\newtheorem{thm}{Theorem}
\newtheorem{lemma}[thm]{Lemma}

\newtheorem{conj}[thm]{Conjecture}

\theoremstyle{definition}

\theoremstyle{remark}

\newtheorem*{ex}{Example}

\numberwithin{equation}{section}
\numberwithin{thm}{section}

\DeclareMathOperator{\Li}{Li}

\newcommand{\CC}{\mathbb{C}}

\title{On Some Continued Fractions and Divergent Series Arising From Integral Families}
\author{Ishan Joshi}

\begin{document}

\begin{abstract}
    In this paper we present a method to derive Eulerian continued fractions arising from a sequence of integrals. As examples, through a new derivation, we reproduce classical continued fraction expansions for $\log(2)$, the Riemann zeta function $\zeta(s)$, and polylogarithms, while also obtaining several new identities. Finally, we apply the method to construct a divergent continued fraction, which provides a natural assignment of the Euler–Mascheroni constant $\gamma$ as the sum of a particular divergent series through a new summation method which we propose.
\end{abstract}

\maketitle

\section{Introduction}
The theory of continued fractions truly began with the Euclidean algorithm. However, it was not until 1579 when the theory began to take off with the work of Bombelli. The field continued with the works of mathematicians such as Gauss, Euler, and Ramanujan. Continued fractions are a very powerful tool. They often converge faster than series, and thus have powerful applications in approximation theory. They are also often used in irrationality proofs. In this paper, we present a method to generate a continued fraction from a given integral. As a result we have rederived several continued fractions, which can also be derived through Euler's Continued Fraction Theorem. The novelty of this work lies in our method and how it allows us to easily and naturally derive several new identities, such as:  $$1-\frac{49}{6\pi^2} = \dfrac{1}{5 - \dfrac{16}{13- \ddots }} \cdot \dfrac{4}{13- \dfrac{81}{25 - \ddots } }  \cdot \dfrac{9}{25 - \dfrac{256}{41 -\ddots}  }$$

After discussing the general method, we discuss some examples which lead to well known Euler Continued fractions. These include, the natural logarithm, the zeta function, and polylogarithms. Finally we will discuss the most striking application of our work, where our method produces a fraction which diverges to $-1$. This allows us to give meaning to the divergent sum: 

$$\gamma = \frac{7}{12} + \frac{1}{2}\sum_{k=1}^\infty\frac{B_{2k+2}}{k+1}$$

This highlights the beginnings of a new summation method to evaluate divergent sums which can be expanded upon in later works.

\section{Preliminaries and Notation}
A generalized continued fraction is a fraction of the form:
$$\frac{a_0}{b_0 + \dfrac{a_1}{b_1 + \dfrac{a_2}{b_2+ \dfrac{a_3}{b_3 + \cdots } } } }$$

This can often be space consuming to write. A more compact notation for this fraction is: $$\frac{a_0}{b_0 +} \frac{a_1}{b_1+} \frac{a_2}{b_2 +}\frac{a_3}{b_3 +} \cdots = \textbf{K}_{n=0}^\infty \frac{a_n}{b_n}$$

A continued fraction is considered simple if for all $i$, $a_i = 1$. However, throughout this paper, we will refer to generalized continued fractions as "continued fraction" where $a_i,b_i$ are complex numbers. 

\begin{thm}[Euler Continued Fraction Theorem]
    Let $\{a_i\}_{i=1}^\infty$ be a sequence of complex numbers. Then, $$1+\sum_{j=1}^\infty \prod_{i=1}^j a_i ={\displaystyle {\cfrac {1}{1-{\cfrac {a_{1}}{1+a_{1}-{\cfrac {a_{2}}{1+a_{2}-{\cfrac {a_{3}}{1+a_{3}-\ddots }}}}}}}}\,}. $$
\end{thm}

This theorem can be easily proved using induction. A proof can be found in \cite{euler2012introduction}.

\section{Continued Fractions Representing Certain Integrals}
Many functions and constants have natural representations as integrals. These integrals can then be used to create recursions. Often, continued fractions arise from a particular recursion. We will now proceed to show how families of integrals can be used to generate recurrences to create continued fractions. 

Let $f(x)$ be a complex valued integrable function such that $$\int_a^b f(x) dx $$ converges absolutely. Then, let $g(x,n)$ be some function such that there exists some integer $c$ such that $g(x,c) = f(x)$.
Then, let 
$$I_n = \int_a^b g(x,n) dx$$ If $k_1I_n + k_2I_{n+1} = \mu_n$ where $k_1,k_2 \in \CC$, we have that $$\frac{k_1I_{n} + k_2I_{n+1}}{k_1I_{n+1} + k_2I_{n+2}} = \frac{\mu_n}{\mu_{n+1}}$$
Then, if we define $r_n = \frac{I_{n+1}}{I_n}$, dividing the numerator and denominator by $I_{n+1}$ results in
$$\frac{\frac{k_1}{r_n} + k_2}{k_1 + k_2r_{n+1}} = \frac{\mu_n}{\mu_{n+1}}$$
Finally, separating the $r_n$ term, we get that: 
$$r_n = \frac{k_1\mu_{n+1}}{k_1\mu_n - k_2\mu_{n+1} + k_2\mu_n r_{n+1}}$$

This recurrence generates a continued fraction. However, it turns out that fractions generated in this way may not always converge to $r_n = \frac{I_{n+1}}{I_n}$. Examining the continued fraction we can see that $-\frac{k_1}{k_2}$ is a fixed point.

\begin{thm}
    If it exists, the continued fraction for $r_n$ converges if $$1+\sum_{k=1}^\infty \prod_{j=0}^{k-1}-\frac{k_2 \mu_{j+n+1}}{k_1\mu_{j+n}} = 1+\sum_{k=1}^\infty(-1)^k \bigg(\frac{k_2}{k_1}\bigg)^k\frac{\mu_{k+n}}{\mu_n}$$ converges.
\end{thm}

\begin{proof}
    Dividing by $k_1\mu_{n}$, we get $$r_n=\frac{\frac{\mu_{n+1}}{\mu_n}}{1 - \frac{k_2\mu_{n+1}}{k_1\mu_{n}} + \frac{k_2}{k_1}r_{n+1}}$$

    Then, by Euler's continued fraction theorem, $$\frac{1}{1+\frac{k_2}{k_1}r_n} = 1+\sum_{k=1}^\infty \prod_{j=0}^{k-1}-\frac{k_2 \mu_{j+n+1}}{k_1\mu_{j+n}}$$
\end{proof}

Because of this, from here on, we will say that the fraction diverges if it approaches the fixed point $-\frac{k_1}{k_2}$.

\section{The Natural Logarithm}

Now that we have introduced our method, we will use it in a number of examples. A constant which appears frequently in mathematics is $\log(2)$. To give a simple example of our method, we will find a continued fraction for $\log(2)$. Recall Euler's famous continued fraction for $\log(2)$:

$$\log(2) = \frac{1}{1 + \dfrac{1}{1 + \dfrac{2^2}{1+ \dfrac{3^2}{1 + \ddots } } } }$$

In this section, we will re-derive this result using our method, and then state some original and new identities which arise. 

\begin{thm}
    $$\log\bigg(1+\frac{1}{k}\bigg) = \frac{k}{k+ \textbf{K}_{n=1}^\infty \frac{k n^2 }{k(n+1) - n}}$$
\end{thm}

\begin{proof}
    Let $$I_n = \int_0^1 \frac{x^n}{k+x} dx $$ This particular family of integrals is chosen because $I_0 = \log(1+\frac{1}{k})$, and it allows for the nice cancellation which will be demonstrated shortly. Following the method, $$kI_n + I_{n+1} = \int_0^1 \frac{kx^n + x^{n+1}}{k+x}dx = \int_0^1 \frac{x^n(k + x)}{k+x}dx = \int_0^1 x^n dx = \frac{1}{n+1}$$ Then, we have that $$\frac{kI_n + I_{n+1}}{kI_{n+1} + I_{n+2}} = \frac{n+2}{n+1}$$

    Setting $r_n = \frac{I_{n+1}}{I_n}$, and dividing the numerator and denominator by $I_{n+1}$ gives us $$\frac{\frac{k}{r_n} + 1}{k+r_{n+1}} = \frac{n+2}{n+1}$$ $$r_n = \frac{k(n+1)}{k(n+2) - (n+1)+(n+2)r_{n+1}}$$

    Then, $r_0 = \frac{I_1}{I_0} = \frac{1}{\log(1+\frac{1}{k})} - k$, and isolating $\log(1+1/k)$ gives our result.
  
\end{proof}

Now that we have this fraction, substituting $k=1$ gives us Euler's famous continued fraction. Now, we will discuss some identities which arise from this. Immediately we notice that when we multiply the $r_k$'s, we get the following identity. 

\begin{thm}
    $$\int_0^1 \frac{x^n}{k+x} = \log(1+1/k)\prod_{k=0}^nr_k$$
\end{thm}

\begin{proof}
    $$\prod_{k=0}^nr_k=\prod_{k=0}^n \frac{I_{k+1}}{I_k} = \frac{I_{n+1}}{I_0} = \frac{\int_0^1 \frac{x^n}{k+x}}{\log(1+ 1/k)}$$
\end{proof}

The product of the $r_k$'s telescopes, meaning we get an expression for the integral as a product of continued fractions.

\section{The Riemann Zeta Function}

The Riemann Zeta function is an important function in number theory, specifically in analytic number theory. The zeta function is defined as such: $$\zeta(s) = \sum_{n=1}^\infty \frac{1}{n^s} $$ The sum converges when $\Re(s) > 1$, so the function is defined to be its analytic continuation everywhere else. 

There are some other important variants of $\zeta(s)$. These are the truncated zeta function $\zeta_n(s) = \sum_{k=1}^n \frac{1}{k^s}$, and the Hurwitz zeta function $\zeta(s,n) = \sum_{k=0}^\infty  \frac{1}{(k+n)^s}$. We will discuss these later within this section. 
To find a continued fraction representation for $\zeta(s)$, we will first consider this well known integral:
\begin{lemma}
    $$\int_0^\infty \frac{x^{s-1}}{e^x - 1} dx = \zeta(s)\Gamma(s)$$
\end{lemma}
\begin{proof}
    To begin, note that $$\frac{1}{e^x -1} = \frac{e^{-x}}{1-e^{-x}} = \sum_{k=1}^\infty e^{-kx}$$
    
    Then,
    $$\int_0^\infty \frac{x^{s-1}}{e^x-1} dx = \int_0^\infty x^{s-1} \sum_{k=1}^\infty e^{-kx}\ dx = \sum_{k=0}^\infty \int_0^\infty x^{s-1}e^{-kx}\  dx  $$
    $$= \sum_{k=0}^\infty \frac{1}{n^s} \int_0^\infty u^{s-1} e^{-u} \ du = \sum_{k=0}^\infty \frac{1}{n^s} \Gamma(s) = \zeta(s)\Gamma(s)$$
    Note that the switching of the sum and integral is allowed when $\Re(s) > 1$ by Fubini's theorem. 
\end{proof}

This classical integral representation for the zeta function gives us enough to find a continued fraction representation for $\zeta(s)$. 

\begin{thm}
    For all $s \in \CC$, where $\Re(s) > 1$, $$1-\frac{1}{\zeta(s)} = -\textbf{K}_{n=1}^\infty \frac{-n^{2s}}{n^s + (n+1)^s}$$
\end{thm}
\begin{proof}
    Consider the integral $$I_n = \int_0^\infty \frac{e^{-nx} x^{s-1}}{e^x-1} dx$$
    We see from above that $I_0 = \zeta(s)\Gamma(s)$. Similarly to before, this family of integrals was chosen in order to obtain the required cancellation. Then, 

    $$I_{n-1} - I_n = \int_0^\infty \frac{e^{-nx + x }x^{s-1}}{e^x -1}dx - \int_0^\infty \frac{e^{-nx + x }x^{s-1}}{e^x -1}dx = \int_0^\infty \frac{e^{-nx}x^{s-1}(e^x -1)}{e^x -1}dx$$ $$=\int_0^\infty e^{-nx}x^{s-1} dx = \frac{1}{n^s}\int_0^\infty e^{-u} u^{s-1} du = \frac{1}{n^s} \Gamma(s)$$

    Then, proceeding with the next step in the method, we get $$\frac{I_n - I_{n+1}}{I_{n+1}-I_{n+2}} = \frac{\frac{1}{(n+1)^s} \Gamma(s)}{\frac{1}{(n+2)^s}\Gamma(s)} = \frac{(n+2)^s}{(n+1)^s}$$

    Letting $r_n = \frac{I_{n+1}}{I_n}$ and dividing the numerator and denominator by $I_{n+1}$, we find that $$\frac{\frac{1}{r_n} -1 }{1-r_{n+1}} = \frac{(n+2)^s}{(n+1)^s}$$ Thus, $$r_n = \frac{(n+1)^s}{(n+2)^s + (n+1)^s - (n+2)^s\  r_{n+1}}$$ This gives us the recurrence relation for our continued fraction. Recall that $I_0 = \zeta(s) \Gamma(s)$. By evaluating $I_1$ using the exact same method, we see that $I_1 = (\zeta(s) -1)\Gamma(s)$. So, $$r_0 = \frac{(\zeta(s) - 1)\Gamma(s)}{\zeta(s)\Gamma(s)}= 1-\frac{1}{\zeta(s)}$$ Writing the recurrence relation in continued fraction form gives us that $$r_0 = 1-\frac{1}{\zeta(s)} = -\textbf{K}_{n=1}^\infty \frac{-n^{2s}}{n^s + (n+1)^s}$$
\end{proof}

Therefore, we clearly see that $$\zeta(s) = \frac{1}{1+
\textbf{K}_{n=1}^\infty \frac{-n^{2s}}{n^s + (n+1)^s}}$$

This is precisely the fraction that results after applying the Euler Continued Fraction theorem on the series definition for $\zeta$. Our method is a new derivation of this fraction. Additionally, the way in which we find this fraction yields some more results.

\begin{thm}
    $$\frac{\zeta_{n+1}(s)}{\zeta(s)} = 1-\prod_{k=0}^n r_k$$
\end{thm}
\begin{proof}
    $$\prod_{k=0}^n r_k = \prod_{k=0}^n\frac{I_{k+1}}{I_k} = \frac{I_{n+1}}{I_0} = \frac{\Gamma(s)(\zeta(s) - \zeta_{n+1}(s))}{\Gamma(s)\zeta(s)} = 1 - \frac{\zeta_{n+1}(s)}{\zeta(s)}$$
\end{proof}

Letting $s=n=2$ in this formula gives the evaluation mentioned at the beginning of the paper. Directly following from this, we have that $$\zeta_{n+1}(s) = \zeta(s) - \zeta(s)\prod_{k=0}^n r_k$$ Then, since $\zeta(s) - \zeta_{n+1}(s) = \zeta(s,n+2)$, we get that $$\zeta(s,n+2) = \zeta(s)\prod_{k=0}^n r_k$$

Next, we will discuss another identity which also easily arises.

\begin{thm}
    $$\zeta(s, n+1) = r_n^k\zeta(s,n+1) + \frac{1}{(n+1)^s}\sum_{j=1}^k r^j_k$$
\end{thm}

\begin{proof}
    $$r_n = \frac{I_{n+1}}{I_n} = \frac{\zeta(s) - \zeta_{n+1}(s)}{\zeta(s) - \zeta_n(s)} $$
    Then $r_n(\zeta(s) - \zeta_n(s)) + \frac{1}{(n+1)^s} = \zeta(s) - \zeta_n(s)$ Repeating this $k$ times, we get $$\zeta(s) - \zeta_n(s) = r_n^k( \zeta(s) - \zeta_n(s)) + \frac{1}{(n+1)^s}\sum_{j=1}^k  r_n^j$$

    Then, $$\zeta(s, n+1) = r_n^k\zeta(s,n+1) + \frac{1}{(n+1)^s}\sum_{j=1}^k r^j_n$$

\end{proof}

Thus, we have expressed the Hurwitz zeta function as a a sum of itself scaled by a continued fraction, added to a sum of continued fractions.

\section{Polylogarithms}

The polylogarithm, $\Li_s $, is a special function which is defined to be $$\Li_s(z) = \sum_{n=1}^\infty \frac{z^n}{n^s}$$

The zeta function is a special case of the polylogarithm, where $z=1$. The polylogarithm can also be expressed as the integral $$\frac{1}{\Gamma(s)}\int_0^\infty \frac{x^{s-1}}{\frac{e^x}{z} - 1}dx$$

\begin{thm}
    $$\Li_s(z) = \frac{1}{\frac{1}{z} - \frac{1}{z + 2^s +\textbf{K}_{n=2}^\infty \frac{-zn^{2s}}{zn^s + (n+1)^s} } }$$
\end{thm}

The proof for this fact is almost identical to the proof for the continued fraction for the zeta function. 

\begin{proof}
   Let $I_n = \int_0^\infty \frac{x^{s-1 }e^{-nx} }{e^x/z - 1} dx$. Then, we have that $$\frac{1}{z}I_{n-1} - I_n = \int_0^\infty \frac{x^{s-1} e^{-nx} ( e^x/z - 1)}{e^x/z - 1}  = \int_0^\infty x^{s-1}e^{-nx}dx = \frac{1}{n^s}\Gamma(s)$$

$$\frac{I_n - zI_{n+1}}{I_{n+1} - zI_{n+2}} = \frac{(n+2)^s}{(n+1)^s}$$ $$\frac{\frac{1}{r_n} - z}{1-zr_{n+1}} = \frac{(n+2)^s}{(n+1)^s}$$ $$r_n = \frac{(n+1)^s}{z(n+1)^s + (n+2)^s - z(n+2)^sr_{n+1}}$$

Expanding $r_0$ as a continued fraction gives the desired result. 
\end{proof}

Then, we trivially get the identity $$\prod_{k=0}^n r_k = \prod_{k=0}^n \frac{I_{k+1}}{I_k} = \frac{I_{n+1}}{I_0} = \frac{z^{-n}(\Li_s(z) - \Li_{s, n+1}(z))}{\Li_s(z)}$$

Next, we also get that since $$r_n = \frac{\Li_s(z) - \Li_{s,n+1}(z) }{z(\Li_s(z) - \Li_{s,n}(z))}$$

$$zr_n(\Li_s(z) - \Li_{s,n}(z) ) + \frac{z^{n+1}}{(n+1)^s}$$ Repeating this gives us that: $$\Li_s(z) - \Li_{s,n}(z)  = z^kr_n^k (\Li_s(z) - \Li_{s,n}(z) ) + \frac{z^{n+1}}{(n+1)^s}\sum_{j=1}^k r_n^j$$
$$\Phi(z,s,n+1) = r_n^k  z^{k}\Phi(z,s,n+1) + \frac{z^{n}}{(n+1)^s}\sum_{j=1}^k r_n^j $$

where, $\Phi(z,s,n) = \sum_{k=0}^\infty \frac{z^k}{(k+n)^s}$.

\section{The Euler-Mascheroni Constant}

The Euler-Mascheroni constant is a constant which appears very often in analysis and number theory. The Euler-Mascheroni Constant (Henceforth referred to as $\gamma$) can be defined as such:

$$\gamma = \lim_{n \to \infty} \Bigg(\sum_{k=1}^n \frac{1}{k}  \ - \log(n)\Bigg) \approx 0.57721 $$

Although a continued fraction is known, it appears to have no pattern. In this section, we will derive a continued fraction which diverges, but with our method, can be assigned the value $\frac{1}{12(\gamma -\frac{1}{2})} - 1$. Unlike the other fractions, this fraction is novel and demonstrates how our method can also be used to evaluate divergent infinite sums. In order to derive the continued fraction, as before, we will start with an integral. 

\begin{lemma}
    $$\int_0^\infty \frac{2x}{(x^2+1)(e^{2\pi x} -1)}dx = \gamma -\frac{1}{2}$$
\end{lemma}
A proof can be found in \cite{whittaker2020course}

\begin{lemma}
    $$\int_0^\infty \frac{2x^3}{(x^2+1)(e^{2\pi x} -1)}dx = \frac{7}{12} -\gamma$$
\end{lemma}
\begin{proof}
    Note that $\frac{x^3}{x^2+1} = x - \frac{x}{x^2+1}$
    Then  $$\int_0^\infty \frac{2x^3}{(x^2+1)(e^{2\pi x} -1)}dx = \int_0^\infty \frac{2x}{e^{2\pi x} - 1} dx - \int_0^\infty \frac{2x}{(x^2+1)(e^{2\pi x} -1)}dx $$ $$ = \int_0^\infty \frac{2x}{e^{2\pi x} - 1} dx - \gamma+\frac{1}{2} = \frac{1}{\pi^2}\zeta(2) - \gamma + \frac{1}{2} = \frac{7}{12} - \gamma$$
\end{proof}

\begin{thm}
    The continued fraction for $$\frac{7/12 - \gamma}{\gamma - 1/2} = \frac{1}{12(\gamma -\frac{1}{2})}-1$$ is given by $r_0$ in the recursion $$r_n = \frac{-nB_{2n+2}}{nB_{2n+2} + (n+1)B_{2n} + (n+1)(B_{2n})r_{n+1}}$$
\end{thm}
\begin{proof}

    Let $$I_n = \int_0^\infty \frac{2x^{2n-1}}{(x^2+1)(e^{2\pi x} -1)} dx$$

    From the previous two lemmas, we have that $I_1 = \gamma -\frac{1}{2}$, and $I_2 = \frac{7}{12} - \gamma$.
    We will now consider the sum $I_n + I_{n+1}$. $$I_n+I_{n+1} = \int_0^\infty \frac{2x^{2n -1}}{(x^2+1)(e^{2\pi x }-1)}dx + \int_0^\infty \frac{2x^{2n+1}}{(x^2+1)(e^{2\pi x}-1)}dx $$$$= 2\int_0^\infty \frac{x^{2n-1}(1+x^2)}{(x^2+1)(e^{2\pi x} -1)}dx = 2\int_0^\infty \frac{x^{2n-1}}{e^{2\pi x}-1 }dx = 2\bigg[\frac{(2n-1)! \zeta(2n)}{(2\pi)^{2n}}\bigg] $$

    Then, we have that $$\frac{I_n+I_{n+1}}{I_{n+1}+I_{n+2}} = \frac{\frac{(2n-1)!\zeta(2n)}{(2\pi)^{2n}}}{\frac{(2n+1)! \zeta(2n+2)}{(2\pi)^{2n+2}}}  = \frac{4\pi^2 \zeta(2n)}{(2n)(2n+1)\zeta(2n+2)}$$

    Letting $r_n = \frac{I_{n+1}}{I_n}$ and dividing the numerator and denominator of the fraction by $I_{n+1}$, we get that $$\frac{\frac{1}{r_n} +1}{1+r_{n+1}} = \frac{4\pi^2 \zeta(2n)}{(2n)(2n+1)\zeta(2n+2)} $$ Therefore, $$r_n = \frac{(2n^2 + n) \zeta(2n+2)}{(-2n^2 - n)\zeta(2n+2) + 4\pi^2 \zeta(2n) + 4\pi^2 \zeta(2n)r_{n+1}}$$

    However, we know something about even values of the zeta function. Specifically, $$\zeta(2n) = (-1)^{n+1} \frac{B_{2n }(2\pi)^{2n}}{2(2n)!}$$

    Then, once we simplify the expression, we get that $$r_n = \frac{-nB_{2n+2}}{nB_{2n+2}+(n+1)B_{2n} + (n+1)(B_{2n})r_{n+1}}$$

    We also have that $$r_1 = \frac{I_{2}}{I_1} = \frac{\frac{7}{12}-\gamma}{\gamma - \frac{1}{2}} = \frac{1}{12(\gamma-\frac{1}{2})}-1$$

\end{proof}

Cleaning this up a little, we get that $$r_n = \frac{-1}{1+\frac{(n+1)B_{2n}}{n B_{2n+2}} + \frac{(n+1)B_{2n}}{n B_{2n+2}}r_{n+1}}$$ 

Then $$r_1 = \frac{-1}{b_1 - \frac{a_1}{b_2- \dots}}$$ $$b_n = {1+\frac{(n+1)B_{2n}}{n B_{2n+2}}}, \ \ \ \ \ \ a_n = \frac{(n+1)B_{2n}}{n B_{2n+2}} =b_n-1 $$

By examining this fraction, we can see that it does not converge to the $r_1$ but instead to $-1$. However, by now applying Euler's Continued Fraction theorem, we arrive at the divergent series evaluation: $$\gamma = \frac{7}{12} + \frac{1}{2}\sum_{k=1}^\infty \frac{B_{2k+2}}{k+1}$$

By noticing that $\zeta(-n) = -\frac{B_{n+1}}{n+1}$, we get that $$\gamma = \frac{7}{12} - \sum_{k=1}^\infty \zeta(-2k-1)$$

This is a new evaluation of this divergent series, and demonstrates the applicability of our method even in cases where the fraction itself diverges. 

\section{More on C-Summation}
This section will serve to elaborate on the summation method which emerged in the previous section, which is henceforth referenced as C-Summation (short for continued fraction summation). In this section, we discuss some properties of C-Summation, some problems which arise, and compare it to other summation methods. To first convince the reader of it's applicability, we will use it to reproduce a famous result by Euler, which assigns a value to the alternating sum of factorials. \begin{ex}
    Suppose $$I_n = \int_0^\infty\frac{x^ne^{-x}}{1+x} dx$$ Then $I_{n} + I_{n+1} = \int_0^\infty x^n e^{-x} dx = n! $. The corresponding sum would be $$\int_0^\infty \frac{e^{-x}}{1+x}dx= 1+\sum_{k=1}^\infty (-1)^k {k!} = 1 - 1+2-6 + 24 - 120 \cdots $$

    This is precisely the sum which Euler gave to this series. However, we can extend this. Letting $s\in \CC$, we can assign $$\frac{1}{\Gamma(s+1)}\int_0^\infty \frac{x^s e^{-x}}{1+x}dx = 1+\sum_{k=1}^\infty (-1)^k\frac{\Gamma(s+k+1)}{\Gamma(s)}$$ Therefore, we have the following divergent series:  $$\int_0^\infty \frac{x^s e^{-x}}{1+x}dx = \sum_{k=0}^\infty (-1)^k \ \Gamma(s+k+1) = \Gamma(s+1) - \Gamma(s+2) +\Gamma(s+3) \cdots $$
\end{ex}

This evaluation is exactly what Euler got, meaning in this case C-Summation aligns with Euler Summation. However, before we continue to examine C-Summation, we must address a major problem, which is uniqueness. To illustrate this we present an example.

\begin{ex}
    Suppose $$I_n = \int_0^1 \frac{x^n}{1+x} dx \ \ \ \ \  J_n = \int_0^1 \bigg(\frac{x^n}{1+x} + (-1)^n\bigg) dx$$ It can be easily seen that $I_n+I_{n+1} = J_n + J_{n+1} = \frac{1}{n+1}$. So both of these sequences "correspond" to the same continued fraction and the same sum. However, it is clear that $I_n$ is a better choice and will correctly correspond to the continued fraction. 
\end{ex}

In order to prevent this and chose the correct $I_n$, we must have a criterion. First we recall the equation $k_1I_n + k_2I_{n+1} = \mu_n$. If $I_n$ is a particular solution to the equation, then if $J_n$ also satisfies the equation, then if $H_n = J_n - I_n$, we have that $k_1H_n + k_2 H_{n+1} = 0$. Then $H_{n+1} = -\frac{k_1}{k_2}H_n$, meaning that for some constant $C$, $H_{n+1} = H_1\big(-\frac{k_1}{k_2}\big)^{n-1} $. Then, $J_n =I_n + (I_1-J_1)\big(-\frac{k_1}{k_2}\big)^{n-1} $. In general $J_n =I_n + (I_k-J_k)\big(-\frac{k_1}{k_2}\big)^{n-k} $. Therefore it is necessary that some condition be added which ensures that ensures the right $I_n$ is chosen. When the corresponding series converges, this is easy. We can simply require that $\big(\frac{k_2}{k_1}\big)^{M+1}I_{n+M+1} \to 0  $. However, this does not help us in determining which $I_n$ are correct to use in the divergent case.

\begin{conj}
   If $$\sum_{k=n}^\infty I_kz^k $$ is analytic at $-\frac{k_2}{k_1}$, then $I_n$ is suitable for C-Summation.
\end{conj}

It is essential that we prove some properties of C-Summation. Typically, some properties desirable in a summation method are regularity, linearity, and stability. It turns out that C-Summation is not linear. However, it is stable and regular. We will begin by proving stability.

\begin{lemma}
    C-Summation is stable. 
\end{lemma}

\begin{proof}
    Suppose $s=\frac{k_1I_n}{\mu_n}$ is the value assigned to the C-Summable series $$\sum_{k=0}^\infty (-1)^k \bigg(\frac{k_2}{k_1}\bigg)^k \frac{\mu_{n+k}}{\mu_n}$$ Then, $$k_1I_n = \sum_{k=0}^\infty (-1)^k \bigg(\frac{k_2}{k_1}\bigg)^k \mu_{k+n}$$ 

    Similarly, by the same method, also have $$ s' = k_1I_{n+1} =  \sum_{k=0}^\infty (-1)^k \bigg(\frac{k_2}{k_1}\bigg)^k \mu_{n+k+1} $$

    From this we notice that $$\mu_n + \sum_{k=1}^\infty (-1)^k \bigg(\frac{k_2}{k_1}\bigg)^k \mu_{k+n}  = \mu_{n} - \frac{k_2}{k_1}k_1I_{n+1} = k_1I_n$$
\end{proof}

\begin{lemma}
C-Summation is regular when $\big(\frac{k_2}{k_1}\big)^{M+1}I_{n+M+1} \to 0  $
\end{lemma}

\begin{proof}
    Suppose $$k_1s=\sum_{k=0}^\infty (-1)^k \bigg(\frac{k_2}{k_1}\bigg)^k {\mu_{n+k}}$$ is a convergent sequence. Then $\mu_k = k_1I_n + k_2I_{n+1}$, where $I_n$ is admissible by theorem 8.1 with $p=0$. Then we also have $k_1I_n = \mu_k - \frac{k_2}{k_1} k_1I_{n+1}$ Using this identity $M$ times, we get that $$k_1I_n = \bigg(\sum_{k=1}^M (-1)^k\bigg(\frac{k_2}{k_1}\bigg)^k \mu_{n+k}\bigg) + (-1)^k\bigg(\frac{k_2}{k_1}\bigg)^{M+1}k_1I_{n+M+1}$$ Then, since $\big(\frac{k_2}{k_1}\big)^{M+1}I_{n+M+1} \to 0 $.

\end{proof}

\section{Conclusion}
In this paper, we introduce a general method to evaluate Eulerian continued fractions from families of integrals. This method allows us to not only find the value of the fraction, but also allows us to trivially derive an array of identities from it. The most important application of the technique is evident in the case of the Euler-Mascheroni constant. We see that the continued fraction produced by this method diverges to the fixed point $-1$. However, by then using Euler's continued fraction theorem regardless, we can assign the value $\gamma$ to a divergent sum. This shows that the technique has the possibility to provide a natural way to assign values to divergent sums, when the continued fraction generated does not converge as we expect. Furthermore, the sum assigned to convergent series agrees with the value of the sum. In future work, we may explore this result more, and determine the class of functions where this summation method applies.

\section{Acknowledgments}
The author would like to thank Professor Ovidiu Costin for his encouragement, comments, and suggestions.

\bibliographystyle{alpha}
\bibliography{2025contfracmethod}
\cite{euler2012introduction}
\cite{compton1973integral}
\cite{whittaker2020course}
\cite{pilehrood2013continuedfractionexpansioneulers}
\cite{cuyt2008handbook}
\cite{c39be6f6-2bdf-30c1-bebf-7068a12748e0}
\cite{cvijovic1997continued}
\end{document}